\documentclass[12pt, reqno]{amsart}
\usepackage{amsmath, amsthm, amscd, amsfonts, amssymb, graphicx, color}
\usepackage[bookmarksnumbered, colorlinks, plainpages]{hyperref}
\hypersetup{colorlinks=true,linkcolor=red, anchorcolor=green, citecolor=cyan, urlcolor=red, filecolor=magenta, pdftoolbar=true}

\newtheorem{theorem}{Theorem}[section]
\newtheorem{lemma}[theorem]{Lemma}
\newtheorem{prop}[theorem]{Proposition}
\newtheorem{cor}[theorem]{Corollary}
\theoremstyle{definition}
\newtheorem{defn}{Definition}

\theoremstyle{remark}
\newtheorem{remark}[theorem]{\bf{Remark}}
\numberwithin{equation}{section}
\begin{document}
	
	\title[On local preservation of orthogonality]{On local preservation of orthogonality and its application to isometries}
		\author[Sain, Manna and  Paul  ]{Debmalya Sain, Jayanta Manna and Kallol Paul }

	\newcommand{\acr}{\newline\indent}
\address[Sain]{Department of Mathematics\\ Indian Institute of Information Technology, Raichur\\ Karnataka 584135 \\INDIA}
\email{saindebmalya@gmail.com}

	\address[Manna]{Department of Mathematics\\ Jadavpur University\\ Kolkata 700032\\ West Bengal\\ INDIA}
	\email{iamjayantamanna1@gmail.com}

	\address[Paul]{Department of Mathematics\\ Jadavpur University\\ Kolkata 700032\\ West Bengal\\ INDIA}
	\email{kalloldada@gmail.com}

	\begin{abstract}
		We investigate the local preservation of Birkhoff-James orthogonality at a point by a linear operator on a finite-dimensional Banach space and illustrate its importance in understanding the action of the operator in terms of the geometry of the concerned spaces. In particular, it is shown that such a study is related to the preservation of k-smoothness and the extremal properties of the unit ball of a Banach space. As an application of the results obtained in this direction, we obtain a refinement of the well-known Blanco-Koldobsky-Turnsek characterization of isometries on some polyhedral Banach spaces, including  $ \ell_{\infty}^n, \ell_1^n. $
		
	\end{abstract}


        \maketitle

	\section{Introduction.}
 One of the high points of the isometric theory of Banach spaces is the following fundamental result due to Koldobsky, Blanco, and Turnsek \cite{BT06,K93}:\\
 
 \emph{Every bounded linear operator on a Banach space preserving orthogonality is an isometry multiplied by a constant.}\\
 
 We would like to mention that the first proof of this result was presented by Koldobsky in Theorem 1.5 of \cite{K93}, in the real case. Later Blanco and Turnsek \cite{BT06} obtained a more general proof, covering both the real and the complex cases. Clearly, the above result is concerned with the \emph{global preservation of orthogonality} by bounded linear operators. The purpose of this article is to consider a \emph{local} version of this concept and to obtain some geometric consequences arising out of it. As an application of this study, we prove that in many classical finite-dimensional Banach spaces, including the Euclidean space $ \ell_2^n $ and the polyhedral spaces $ \ell_{\infty}^n, \ell_1^n, $ the Blanco-Koldobsky-Turnsek characterization of isometries (up to multiplication by a constant) can be substantially refined.\\
 
 Let us first fix the notations and the terminologies to be used throughout this article. Letters $ \mathbb{X}, \mathbb{Y} $ denote Banach spaces and the letter $ \mathbb{H} $ denotes a Hilbert space. The unit ball and the unit sphere of the Banach space $ \mathbb{X} $ are denoted by $ B_{\mathbb{X}} $ and $ S_{\mathbb{X}}, $ respectively. The convex hull of a nonempty set $S\subset \mathbb X$ is denoted by co(S). For $x\in\mathbb{X}$, where $X$ is an n-dimensional Banach space, the notation $[ x]_{\mathcal{B}} =(x_1,x_2,\dots,x_n)^t$ denotes that $x_1,x_2,\dots,x_n$ are the coordinates of $x$ relative to an ordered basis $\mathcal{B}$ of $\mathbb{X}$. Let $ \mathbb{L} (\mathbb{X}, \mathbb{Y}) $ be the Banach space of all bounded linear operators from $ \mathbb{X} $ to $ \mathbb{Y}, $ endowed with the usual operator norm. Whenever $ \mathbb{X} = \mathbb{Y}, $ we simply write $ \mathbb{L}(\mathbb{X}, \mathbb{Y}) = \mathbb{L}(\mathbb{X}). $ The identity operator on an $ n $-dimensional Banach space is denoted by $ I_n. $ $ T \in \mathbb{L}(\mathbb{X}, \mathbb{Y}) $ is said to be an isometry if $ \| Tx \| = \| x \| $ for all $ x \in \mathbb{X}. $ One of the fundamental objective of the isometric theory of Banach spaces is to characterize the isometries on $ \mathbb{L}(\mathbb{X}). $ The Blanco-Koldobsky-Turnsek Theorem illustrates the importance of the preservation of orthogonality by bounded linear operators in the whole scheme of things. It should be noted that although there are several distinct notions of orthogonality in the normed setting, the authors in \cite{B35,J47} specifically consider Birkhoff-James orthogonality which is perhaps the most important orthogonality notion in the framework of Banach spaces. Given $x, y \in \mathbb{X}, $ we say that $ x $ is Birkhoff-James orthogonal to $y,$ written as $ x \perp_B y,$ if $ \| x + \lambda y \| \geq \|x\|$ for all scalars $\lambda. $ For any $ x \in \mathbb{X}, $ we use the notation $ x^{\perp_B} := \{ y \in \mathbb{X} : x \perp_B y \}. $ It is clear that the Birkhoff-James orthogonality relation is homogeneous, i.e., $ x \perp_B y $ implies that $ \alpha x \perp_B \beta y $ for any scalars $ \alpha, \beta. $ Moreover, in an inner product space, it is easy to observe that the Birkhoff-James orthogonality $ \perp_B $ coincides with the usual inner product orthogonality $ \perp. $ Given any two subsets $ C, D $ of $ \mathbb{X}, $ we say that $ C \perp_B D $ if $ u \perp_B v $ for all $ u \in C $ and for all $ v \in D. $ $  T \in \mathbb{L}(\mathbb{X}, \mathbb{Y}) $ is said to preserve orthogonality at $ x \in \mathbb{X} $ if for any $  y \in \mathbb{X},$ $ x \perp_B y \Rightarrow Tx \perp_B Ty. $ Birkhoff-James orthogonality is extremely useful in studying many local geometric properties of a Banach space, including smoothness and $ k $-smoothness. We recall that a non-zero $ x \in \mathbb{X} $ is said to be a $ k $-smooth point if $dim~span ~J(x)=k,$ where $ J(x) := \{ f \in S_{\mathbb{X}^*} : f(x) = \| x \| \}. $ Moreover, $ 1 $-smooth points are simply called smooth points. We would like to mention that there exists a large number of articles devoted to the study of smoothness and $ k $-smoothness in Banach spaces, including \cite{DMP22,KS05,LR07,MPD22,MP20,PSG16,SPMR20,W18}.\\
 
 By virtue of the homogeneity property of Birkhoff-James orthogonality, the   Blanco-Koldobsky-Turnsek Theorem asserts that $ T \in \mathbb{L}(\mathbb{X}) $ is a scalar multiple of an isometry if it preserves Birkhoff-James orthogonality at each point of $ S_{\mathbb{X}}. $  In view of this, it is tempting to ask whether the conditions of the said result can be relaxed. More precisely, we ask the following question: \\
 
 \emph{Which subsets $ A $ of $ S_{\mathbb{X}} $ possess the property that if $ T \in \mathbb{L}(\mathbb{X}) $ preserves Birkhoff-James orthogonality at each point of $ A $ then $ T $ is necessarily a scalar multiple of an isometry? In particular, when $ \mathbb{X} $ is finite-dimensional, can such an $ A $ be finite?}\\
 
 On the other hand, since isometries on a Banach space preserve all the geometric properties of the space, it is natural to look for analogous local geometric properties under the assumption of preservation of Birkhoff-James orthogonality by a bounded linear operator at a given point. Let us also mention here that the local preservation of Birkhoff-James orthogonality has been investigated in the recent past, in connection with operator norm attainment \cite{S20,S18} and singular value decomposition of matrices \cite{SRT21}. In this article, we illustrate that such preservation of Birkhoff-James orthogonality at a point by a bounded linear operator also imposes useful restrictions on the order of smoothness of the corresponding image. This observation turns out to be particularly helpful towards understanding isometries in the spirit of the Blanco-Koldobsky-Turnsek Theorem. 
 
 
  Let us now introduce the following definitions, which allow us to capture the central ideas and the objectives of this article:
 
 \begin{defn}
 	Let $ \mathbb{X} $ be a Banach space. A set $A\subset S_{\mathbb{X}}$ is said to be a Koldobsky-set, for the sake of brevity, a $\mathcal K$-set,  if any operator $T\in \mathbb L(\mathbb X)$ that preserves Birkhoff-James orthogonality at each point of $A,$  is necessarily a scalar multiple of an isometry.
 \end{defn}

\begin{defn}
	Let $ \mathbb{X} $ be a Banach space. A set $A\subset S_{\mathbb{X}}$ is said to be a minimal $\mathcal K$-set  if whenever $B\subset A$ is a $\mathcal K$-set, it follows that $B=A.$
\end{defn}
\begin{defn}
	Let $ \mathbb{X} $ be a Banach space having a finite $\mathcal K$-set. Then the $\mathcal K$-number of $ \mathbb{X} $ is defined by $\min$\{$|A|:A $ is a  $\mathcal K$-set of $\mathbb X$\}. We denote the $\mathcal K$-number of $\mathbb X $ by $\kappa(\mathbb X).$ 
	
\end{defn}

It is clear that obtainment of a  $\mathcal K$-set $ A $ for a Banach space $ \mathbb{X} $ with the property that $ A \subsetneq S_{\mathbb{X}} $  will lead to refinement of the Blanco-Koldobsky-Turnsek Theorem for that particular Banach space $ \mathbb{X}. $ Additionally, if such a $\mathcal K$-set is finite, then it will immediately lead to a significant generalization of the Blanco-Koldobsky-Turnsek Theorem. Moreover, a complete description of minimal $\mathcal K$-sets in $ \mathbb{X} $ will give the optimal versions of the said theorem for $ \mathbb{X}. $ The central aim of this article is to study this idea in the context of finite-dimensional Banach spaces. We begin by considering the Euclidean spaces $ \ell_2^n $. In this case, we are able to obtain a satisfactory answer to our query, by characterizing the minimal $\mathcal K$-sets in geometric terms. The problem is far more complicated in the general setting of Banach spaces. Here we consider the finite-dimensional polyhedral Banach spaces and obtain some partial positive results. We recall that a finite-dimensional real Banach space $ \mathbb{X} $ is polyhedral if $ B_{\mathbb{X}} $ has only finitely many extreme points. We refer the readers to \cite{A05,SPBG19} for more information on the geometric and analytic properties of polyhedral Banach spaces. We obtain refinements of the Blanco-Koldobsky-Turnsek Theorem in the specific cases of $ \ell_{\infty}^n, \ell_1^n. $ In order to proceed in this direction, we present an obstructive property on certain kinds of images of linear maps which is also of independent interest. 

\section{Main Results}
  The results obtained in this article are divided into two sections. In the first section, we focus on the Euclidean spaces and obtain the optimal version of the Blanco-Koldobsky-Turnsek Theorem. In the second section, we consider the more general case of finite-dimensional Banach spaces, with special emphasis on polyhedral spaces such as $ \ell_{\infty}^n $ and $ \ell_1^n. $ It is to be noted that as far as this article is concerned, all polyhedral spaces are considered over the real field.
 
 \section*{Section-I}
	
	We begin this section with an easy observation on the size of $\mathcal K$-sets in finite-dimensional Banach spaces.

	\begin{prop}\label{dokn}
		Let $\mathbb X$ be a Banach space with $dim~\mathbb X=n.$ If $A$ is a $\mathcal K$-set then $dim~span~A=n.$ 
	\end{prop}
	\begin{proof}
		Suppose on the contrary that $A$ is a $\mathcal K$-set with $dim~span~A=r,$ for some $r<n.$ Let $S=\{e_1,e_2,\dots,e_r\}$ be a basis of $span~A.$ Extend $S$ to a basis $\mathcal B=\{e_1,e_2,\dots,e_r,e_{r+1},\dots,e_n\}$ of $\mathbb X.$ Define a linear operator $T$ on $\mathbb X$ as : $Te_i=0, 1\leq i\leq r,$ and $Te_i=e_i, r+1\leq i\leq n.$ Then it is clear that $T$ preserves Birkhoff-James orthogonality at $A$ but $T$ is not a scalar multiple of an isometry. This contradicts the hypothesis that $A$ is a $\mathcal K$-set. Therefore, $dim~span~A=n.$
	\end{proof}
	
	We next obtain a complete characterization of the local preservation of Birkhoff-James orthogonality by an operator on Euclidean spaces.
	
	\begin{prop}\label{evt}
		Let $\mathbb{H}$ be the $n$-dimensional Hilbert space $\ell_2^n.$ Then  $T\in \mathbb L(\mathbb{H})$ preserves  Birkhoff-James orthogonality at a non-zero $x\in\mathbb{H}$ if and only if $x$ is an eigenvector of $T^*T.$
	\end{prop}
	\begin{proof}
		First we prove the sufficient part. Let $x$ be an eigenvector of $T^*T$. Then $T^*Tx=\lambda x$ for some $\lambda \in\mathbb{R}.$ Now, for any $y\in\mathbb{H},$ $\langle Tx,Ty\rangle=\langle T^*Tx,y\rangle=\langle \lambda x,y\rangle=\lambda \langle x,y\rangle.$ Thus for any $y\in\mathbb{H},$ $\langle x,y\rangle=0\implies\langle Tx,Ty\rangle=0.$ Therefore, $T$ preserves Birkhoff-James orthogonality at $x.$\\
		Next, we prove the necessary part. Suppose that $T\in\mathbb L(\mathbb{H})$ preserves Birkhoff-James orthogonality at a non-zero $x\in\mathbb{H}.$ Since $ \mathbb{H}$ is smooth, there exists a unique support functional $f$ at $x.$ So $\mathbb H=span~\{x\}\oplus ker(f)$. Hence for each element $y\in\mathbb{H},$ there exists a unique scalar $\alpha$ and a unique element $h\in ker(f)$ such that $y = \alpha x + h$. Since  $T$ preserves orthogonality at $x$ and $\langle x, h\rangle=0,$ it follows that $\langle Tx, Th\rangle =0.$
		Now,
		$\langle Ty, Tx\rangle = \langle T(\alpha x + h), Tx\rangle=\alpha \|Tx\|^2
		=\|Tx\|^2(\langle \alpha x, x\rangle + \langle h, x\rangle)= \|Tx\|^2 \langle \alpha x + h, x\rangle = \|Tx\|^2 \langle y,x\rangle.$
		Hence $\langle Tx, Ty\rangle = \|Tx\|^2 \langle x,y\rangle$ for all $y\in\mathbb{H},$ i.e., $\langle T^*Tx- \|Tx\|^2x, y\rangle =0$ for all $y\in\mathbb{H}.$ Therefore, $x$ is an eigenvector of $T^*T$ corresponding to the eigenvalue $\|Tx\|^2.$ This completes the proof of the proposition.
	\end{proof}

We also require the following fact, which is stated without any proof, since it follows rather easily from elementary linear algebra.

	\begin{prop}\label{smoi}
		Let $\mathbb{X}$ be an $n$-dimensional Banach space and let $T\in \mathbb L(\mathbb{X}).$  If every element in $\mathbb{X}$ is an eigenvector of $T$ then $T$ is a scalar multiple of the identity operator.  
	\end{prop}

We are now ready to completely characterize the minimal $\mathcal K$-sets in Euclidean spaces, thereby obtaining the optimal version of the Blanco-Koldobsky-Turnsek Theorem for such spaces.

	\begin{theorem}
		Let $\mathbb{H}$ be the  n-dimensional Hilbert space $\ell_2^n.$ A nonempty set $A\subset S_{\mathbb{H}}$ is a minimal $\mathcal K$-set if and only if $A$  satisfies the  following two conditions:
		\begin{itemize}
			\item[(i)] $A=\{x_1,x_2,\dots,x_n\},$ where $x_1,x_2,\dots,x_n$ are linearly independent.
			\item[(ii)] If $A=A_1\cup A_2,$ where $A_1\neq\phi, $ $A_2\neq\phi,$ then $A_1 \not \perp_B A_2.$
		\end{itemize}
	\end{theorem}
	\begin{proof}
		We first prove the sufficient part. Suppose that $ A \subset S_{\mathbb{H}} $ satisfies (i) and (ii). Let $T\in \mathbb L(\mathbb{H})$ be such that $T$ preserves Birkhoff-James orthogonality at $A$.  Therefore, it follows from Proposition \ref{evt}  that each element of $A$ is an eigenvector of $T^*T.$ Since any two eigen spaces of $T^*T$ are orthogonal, it follows from (ii) that all the vectors of $A$ must belong to the same eigen space of $T^*T.$ Since by (i), $dim~span~A=n,$ it follows that all the vectors of $\mathbb{H}$ are eigenvectors of $T^*T.$ Applying Proposition \ref{smoi}, we conclude that $T^*T= \lambda I_n,$ for some scalar $ \lambda, $ i.e., $T$ is a scalar multiple of an isometry. As $T \in \mathbb L(\mathbb{H}) $ is chosen arbitrarily, $A$ is a $\mathcal K$-set for $ \mathbb{H}. $ On the other hand, it follows from Proposition \ref{dokn} that no proper subset of $A$ is a $\mathcal K$-set for $ \mathbb{H}. $ Therefore, $A$ is a minimal $\mathcal K$-set.\\
		Next, we prove the necessary part. Let $A\subset S_{\mathbb{H}}$  be a minimal $\mathcal K$-set.  We first show that (ii) holds. Suppose on the contrary that there exist nonempty $A_1,A_2\subset A$ with $A=A_1\cup A_2$, such that $A_1 \perp_B A_2.$ Let $T$ be the orthogonal projection of $\mathbb{H}$ onto $span~A_1.$  Then $M_T=S_{span~A_1}$ and $T(A_2)=\{0\}$. It is easy to check that $T$ preserves orthogonality on ${A}$, but $T$ is not a scalar multiple of an isometry. This contradicts the hypothesis that $A$ is a $\mathcal K$-set. Thus (ii) holds. We next show that (i) holds. It follows from Proposition \ref{dokn} that $dim~span~A=n,$ and therefore, it is clear that $|A| \geq n.$ If possible, suppose that $|A|>n.$  Consider any non-zero $x_1 \in A.$ Let $B_1=\{x_1\}.$ Clearly, $A=(span~B_1\cap A) \cup (A\setminus span~B_1).$ Since (ii) holds, there exists $x_2\in A \setminus span~B_1 $ such that $ B_1\not\perp\{ x_2\}. $ Let $B_2=\{x_1,x_2\}.$ Again A  can be expressed as  $A=(span~B_2\cap A) \cup (A\setminus span~B_2 ) $ and so there exists $x_3\in A \setminus span~B_2 $ such that $ B_2\not\perp\{ x_3\}. $ Let $B_3=\{x_1,x_2,x_3\}.$ Proceeding in this way we can find a set $ B_n=\{x_1,x_2,...,x_n\} \subset A $ such  that $B_n$ satisfies (i) and (ii). Therefore, it follows from the sufficient part that $ B_n $ is a minimal $\mathcal K$-set, contradicting the minimality of $ A. $ This proves that $|A|=n.$ Thus (i) holds and this completes the proof of the theorem.
	\end{proof}

\section*{Section-II}

In this section, we consider local preservation of Birkhoff-James orthogonality on polyhedral Banach spaces and obtain some interesting consequences of such a study. Our main observation is that preservation of orthogonality by an operator at a unit vector imposes a certain restriction on the order of smoothness of the image vector. The proof of this fact relies on a particular property of linear maps on a finite-dimensional vector space. In order to proceed in this direction, we require the following two technical lemmas.
 
 \begin{lemma}\label{lem1}
   	Let $A=\big({a}_{ij}\big)_{n\times n}$be non-singular, where $a_{ij}\in\mathbb{R}, \forall \,1 \leq i,j \leq n.$   Let $ k_i,h_i \in \mathbb{R}, 1\leq i\leq n,$ with $1+ \sum_{i=1}^{n} k_ih_i \neq 0.$  Suppose that $B=\big({b}_{ij}\big)_{n\times n+1},$ where  $b_{ij}=a_{ij}, $ for $1\leq i,j\leq n$ and $b_{i n+1}=\sum_{j=1}^{n} h_ja_{i j},$ for $1\leq i\leq n$.  Then for a given $y\in\mathbb{R}^n,$  the equation $ Bx=y$ is solvable in $x$  with  $x= (x_1,\dots,x_n,\sum_{i=1}^{n} k_ix_i)^t\in\mathbb{R}^{n+1}.$

   \end{lemma}
   \begin{proof}
   	For $x = (x_1,\dots,x_n,\sum_{i=1}^{n} k_ix_i)^t$,
   	\begin{eqnarray*} 
   		Bx &=&
   		\begin{pmatrix}
   			\sum_{i=1}^{n}a_{1i}x_i+ \sum_{j=1}^{n} h_ja_{1 j} \sum_{i=1}^{n} k_ix_i \\ 
   			\sum_{i=1}^{n}a_{2i}x_i+ \sum_{j=1}^{n} h_ja_{2 j}  \sum_{i=1}^{n} k_ix_i   \\  
   			\vdots \\
   			\sum_{i=1}^{n}a_{ni}x_i+ \sum_{j=1}^{n} h_ja_{n j}  \sum_{i=1}^{n} k_ix_i
   		\end{pmatrix} \\
   		&=& C\tilde{x},
   	\end{eqnarray*}\\
   	where $ C=\big({d}_{ij}\big)_{n\times n}, \, d_{ij}=a_{ij}+k_j\sum_{l=1}^{n} h_la_{il},  \forall \, 1 \leq i,j \leq n$ 
   	and $  \tilde x =(x_1,x_2,\dots,x_n)^t.$\\
   	Now,\begin{eqnarray*}
   		det(C)
   		&=& det(A)\cdot
   		\begin{vmatrix}
   			1+k_1h_1 & k_2h_1  & \dots & k_{n}h_1  \\ 
   			k_1h_2   & 1+k_2h_2 & \dots & k_{n}h_2   \\  
   			\vdots 	& \vdots  & \ddots & \vdots \\
   			k_1h_n  & k_2h_n   & \dots   &1+ k_nh_n
   		\end{vmatrix} \\
   		&=& det(A)\cdot(1+k_1h_1+k_2h_2+\dots+k_{n}h_{n}) \\
   		&\not =&  0.
   	\end{eqnarray*}
   	Therefore, there exists a solution $(x_1,x_2,...,x_n)\in\mathbb{R}^{n}$ of the equation
   	$C\tilde{x}=y.$
   	Hence $x=(x_1,x_2,\dots,x_n,\sum_{i=1}^{n}k_ix_i)^t\in\mathbb{R}^{n+1}$ is a solution of the equation 
   	$Bx=y.$
   \end{proof}

   \begin{lemma}\label{lem2}
   	Let $ \mathbb{X} $ be an $n$-dimensional linear space and let $\{e_1,e_2, \ldots e_n \}$ be a basis of $\mathbb{X}.$ Let $ \{f_1, f_2,\dots, f_k \}\subset\mathbb{\mathbb{X^*}}$ be a set of $k (\leq n)$ linearly independent functionals. Consider the set $F=  co(\{f_1,f_2, \ldots,f_k\}). $ Suppose that $m_1,m_2,\dots,m_p\in\{1,2,\dots,n\},$ where $1\leq p<k.$ Then there exist $r\in\{1,2,\dots,n\} \backslash \{m_1,m_2,\dots,m_p\}$ and $\tilde{f}_1,\tilde{f}_2,\dots,\tilde{f}_k\in F$ such that for each  $1\leq i\leq k$ and for each  $x= \sum_{t=1}^{n} c_te_t, c_t \in \mathbb{R}$,
   	$$ \tilde{f}_i(x) =\sum_{j=1}^{n} b_{ij}c_j,$$ 
   	where  $b_{ij}\in\mathbb{R},$   $b_{ir}\neq 0,$ $\tilde{b}_{l}=(b_{1l}, b_{2l} \dots, b_{kl})^t $ for $l\in\{r,m_1,m_2,\dots,m_p\},$ and $\tilde{b}_r\notin span~ \big\{\tilde{b}_{m_1},\tilde{b}_{m_2},\dots,\tilde{b}_{m_p}\big\}.$  
   \end{lemma}
   \begin{proof}
   	Let $ x = \sum_{t=1}^{n} c_te_t,$ where $  c_t \in \mathbb{R}.$  For each $1\leq i\leq k,$ we can write $f_i(x) = \sum_{j=1}^n a_{ij} c_j, $ where $a_{ij} \in \mathbb{R}.$ Since $f_1, f_2,\dots, f_k$ are linearly independent, it follows that  $A=\big({a}_{ij}\big)_{k\times n}$ is of rank $k$. So $A$ has $k$ linearly independent columns. Let $  m_1,m_2,\dots,m_p \in \{1,2,\ldots,n\}$ and  $p<k.$ Then there exists $r \in \{1,2,\dots,n\} \backslash \{m_1,m_2,\dots,m_p\}$ such that $\tilde{a}_r\notin span~\big\{\tilde{a}_{m_1},\tilde{a}_{m_2},\dots,\tilde{a}_{m_p}\big\},$ where
    $\tilde{a}_{l}=(a_{1l}, a_{2l} \dots, a_{kl})^t $  for $l\in\{r,m_1,m_2,\dots,m_p\}.$
   	Clearly, $\tilde{a}_r\neq 0$, i.e., $(a_{1r}, a_{2r} \dots, a_{kr})^t\neq (0,0,\dots,0)^t.$
   	 Consider the sets
   	\begin{eqnarray*}
   		A_1&=&\{f_i:a_{ir}\neq 0, 1\leq i\leq n\} \\ 
   		\text{and } A_2&=&\{f_i:a_{ir}=0,  1\leq i\leq n\}
   	\end{eqnarray*}
   	Since  $(a_{1r}, a_{2r} \dots, a_{kr})^t\neq (0,0,\dots,0)^t,$ it follows that there exists $s\in\{ 1,2,\dots,k\}$ such that $a_{sr}\neq0$. So $f_s\in A_1$. Hence $A_1$ is nonempty. Without loss of generality  assume that  $f_1, f_2,\dots, f_q \in A_1 $ and $f_{q+1}, f_{q+2},\dots, f_k \in A_2,$ for some $1<q\leq k$.
   	Let  \begin{eqnarray*}
   			\tilde{f}_i&=&f_i, \quad \quad \,\,\text{ if } 1\leq i\leq q \\
   			&=&\frac{f_i+f_1}{2},\text{ if } q<i\leq k .
   		\end{eqnarray*}
   	Clearly, $\tilde{f}_i\in F,$ for all $1\leq i \leq k.$
   	For each $1 \leq i\leq k,$  let $ \tilde{f}_i(x) =\sum_{i=1}^{n} b_{ij}c_j .$ Then it is easy to see that
   	\begin{eqnarray*}b_{ij}&=&a_{ij},\quad\quad\,\,\,\,\text{ if } 1\leq i\leq q , 1\leq j\leq n\\
   		&=& \frac{a_{ij}+a_{1j}}{2},\text{ if } q<i\leq k,  1\leq j\leq n.
   	\end{eqnarray*}
   	Hence $b_{ir}\neq 0,$ for all $1\leq i\leq k.$  Now, we claim that $\tilde{b}_r\notin span~ \big\{\tilde{b}_{m_1},\tilde{b}_{m_2},\dots,\tilde{b}_{m_p}\big\}.$\\
   	On the contrary to our claim suppose that there exist $\gamma_1,\gamma_2,\dots,\gamma_p\in\mathbb{R}$ such that
   	$\gamma_1\tilde{b}_{m_1}+\gamma_2\tilde{b}_{m_2}+\dots+\gamma_p\tilde{b}_{m_p}=\tilde{b}_r.$ Since 
   	\begin{eqnarray*}b_{ij}&=&a_{ij},\quad\quad\,\,\,\,\text{ if } 1\leq i\leq q , 1\leq j\leq n\\
   		&=& \frac{a_{ij}+a_{1j}}{2},\text{ if } q<i\leq k,  1\leq j\leq n,
   	\end{eqnarray*} it follows that $\gamma_1\tilde{a}_{m_1}+\gamma_2\tilde{a}_{m_2}+\dots+\gamma_p\tilde{a}_{m_p}=\tilde{a}_r.$ This contradicts the fact that $\tilde{a}_r\notin span~\big\{\tilde{a}_{m_1},\tilde{a}_{m_2},\dots,\tilde{a}_{m_p}\big\}$.
   	Thus our claim is established and this completes the proof.
   \end{proof}
   
   We are now ready to state and prove the desired property of linear maps on a finite-dimensional Banach space. Although purely linear algebraic in nature, we present the result in the setting of Banach spaces, since our main objective is to apply it to the theory of $ k $-smoothness.  
   
   \begin{theorem}\label{th1}
   	Let $ \mathbb{X} $ and $ \mathbb{Y} $ be Banach spaces of dimensions $n$ and $m$ respectively. Let $ \{f_1, f_2, ..., f_k \}$ be a set of $k ( k\leq n)$ linearly independent functionals in $\mathbb{X}^*$ and let $ \{g_1, g_2, ..., g_p \}$ be a set of $p ( p<k,\,p\leq m)$ linearly independent functionals in $\mathbb{Y}^*.$ Let $F= co  (\{f_1, f_2, ..., f_k \})$ and 
   	$$G=\Big \{\sum_{i=1}^{p}\beta_ig_i :\sum_{i=1}^{p}\beta_i=1, \beta_i\in \mathbb{R} \,\, \forall\, i=1,2,\ldots,p \Big \}.$$
   	Let $A=\bigcup\limits_{f\in F}ker(f)$ and  $B=\bigcup\limits_{g\in G}ker(g).$ If $T\in\mathbb{L(\mathbb X,\mathbb Y)}$ is such that $T(A)\subset B$ then $T(\mathbb X)\subset B .$
   \end{theorem}
   \begin{proof}
   	Let  $T\in\mathbb{L(\mathbb X,\mathbb Y)}$ be such that $T(A)\subset B$. Suppose that $\mathcal{B}_1$ and $\mathcal{B}_2$ are ordered bases of $ \mathbb{X} $ and $ \mathbb{Y} $ respectively. Let the matrix representation of $T$ relative to $(\mathcal{B}_1,\mathcal{B}_2)$ be
   	\[L=\big({l}_{ij}\big)_{m\times n}.\]
   	For each $x\in \mathbb X$ with $[ x]_{\mathcal{B}_1} =(x_1,x_2,\dots,x_n)^t$ and for each $1\leq i \leq k,$ let
   	$ f_i(x) =\sum_{j=1}^{n} a_{ij}x_j.$   Also, for each $y\in \mathbb Y$ with 
 $[ y]_{\mathcal{B}_2} =(y_1,y_2,\dots,y_m)^t$ and for each $1\leq i\leq p,$ let $ g_i(y) =\sum_{j=1}^{m} c'_{ij}y_j.$  \\
   	Consider the matrix $C=\big({c'}_{ij}\big)_{p\times m}.$ Then for each $x\in \mathbb X,$ 
   	\[CLu=g,\] 
   	where $u=[ x]_{\mathcal{B}_1} =(x_1,x_2,\dots,x_n)^t$ and $g=(g_1(Tx),g_2(Tx),\dots,g_p(Tx))^t.$
   	Consider the matrix $D=CL=\Big({d}_{ij}\Big)_{p\times n}.$
   	Let the rank of $D$ be $r$ for some $0\leq r\leq p$. If $r=0,$ then for each $x\in\mathbb{X}$, $g_i(Tx)=0$ for all $1\leq i\leq p$ and this implies that for each  $x\in\mathbb{X},$ $Tx\in B.$ Hence $T(\mathbb X)\subset B .$
   	Next, consider  $1\leq r \leq p.$ Suppose that $h_1$-{th}, $h_2$-{th}, \dots, $h_r$-{th} rows of $D$ are linearly independent, where $h_1 <h_2<\dots<h_r\in\{1,2,\dots,p\}.$ If $r<p,$ then for each $i\in \{1,2,\dots,p\}\backslash\{h_1,h_2,\dots,h_r\}$ and for each $j\in \{1,2,\dots,n\}$, $d_{ij}=\sum_{t=1}^{r}\lambda_{ih_t}d_{h_tj}$ for some $\lambda_{ih_t}\in \mathbb{R}.$ So for each $x\in \mathbb X$ with $[ x]_{\mathcal{B}_1} =(x_1,x_2,\dots,x_n)^t$ and for each $i\in \{1,2,\dots,p\}\backslash\{h_1,h_2,\dots,h_r\}$   we have
   	\begin{eqnarray*}
   		g_i(Tx)&=& \sum_{j=1}^{n} d_{ij}x_j\\
   		&=& \sum_{j=1}^{n}\big(\sum_{t=1}^{r}\lambda_{ih_t}d_{h_tj}\big)x_j\\
   		&=& \sum_{t=1}^{r}\lambda_{ih_t}\sum_{j=1}^{n}d_{h_tj}x_j\\
   		&=& \sum_{t=1}^{r}\lambda_{ih_t}g_{h_t}(Tx).
   	\end{eqnarray*}
   	Let $z\in\mathbb{X}.$ We show that $Tz \in B.$ Clearly, it is sufficient to show that there exists $\eta\in A$ such that $g_i(T\eta)=g_i(Tz)$ for all $i\in\{1,2,\dots,p\}.$\\
   	Without loss of generality assume that first $r$ columns of $D$ are linearly independent. It follows from Lemma \ref{lem2} that there exist $s\in\{r+1,r+2,\dots,n\}$ and $\tilde{f}_1,\tilde{f}_2,\dots,\tilde{f}_k\in F$ such that for each  $i\in\{1,2,\dots,k\}$ and for each $x\in \mathbb X$ with $[ x]_{\mathcal{B}_1} =(x_1,x_2,\dots,x_n)^t,$  $ \tilde{f}_i(x) =\sum_{j=1}^{n} c_{ij}x_j,$ where $c_{ij}\in\mathbb{R}$ and $c_{is}\neq 0$ and $\tilde{c}_s\notin span~ \big\{\tilde{c}_1,\tilde{c}_2,\dots,\tilde{c}_r \big\},$ where $\tilde{c}_{l}=(c_{1l}, c_{2l},\dots, c_{kl})^t $ for $l\in\{ s,1,2,\dots,r\}.$
   	Since the first $r$ columns of $D$ are linearly independent and $s\in\{r+1,r+2,\dots,n\},$ it follows that for each $i\in\{1,2,\dots,p\},$ $d_{is}=\sum_{j=1}^{r}\delta_jd_{ij},$ where $\delta_j\in\mathbb R.$
   	Consider the matrix $D_1=\Big(d'_{ij}\Big)_{r\times r}$ where $d'_{ij}=d_{h_ij}$ for $1\leq i,j\leq r.$ Clearly, $det(D_1)\neq0.$
   	 Let $D_2=\Big({d''}_{ij}\Big)_{r\times r+1},$ $d''_{ij}=d'_{ij},$ for $1\leq i,j\leq r,$ $d''_{i r+1}=d_{h_is},$ for $1\leq i\leq r$ and let $w=(g_{h_1}(Tz),g_{h_2}(Tz),\dots,g_{h_r}(Tz))^t.$ Then it follows from Lemma \ref{lem1} that the equation 
   	\begin{eqnarray}\label{eq1}
   		D_2v=w
   	\end{eqnarray} 
   	has  a solution of the form $v=(v_1,\dots,v_r,\sum_{i=1}^{r}\gamma_iv_i)^t\in\mathbb{R}^{r+1},$ if $(1+\sum_{i=1}^{r}\delta_i\gamma_i)\neq 0.$  
   	Next, we claim that there exists $q\in\{1,2,\dots,k\}$ such that for $(\gamma_1,\gamma_2,\dots,\gamma_r)=-(\frac{c_{q1}}{c_{qs}},\frac{c_{q2}}{c_{qs}},\dots,\frac{c_{qr}}{c_{qs}}),$ we have  $(1+\sum_{i=1}^{r}\delta_i\gamma_i)\neq 0.$
   	On the contrary to our claim suppose that 
   	\[ (1-\sum_{j=1}^{r}\frac{c_{ij}}{c_{is}}\delta_j)= 0, \forall i\in\{1,2,\dots,k\}.\] Since  $c_{is}\neq 0$ for all $1\leq i\leq k,$ it follows that $c_{is}=\sum_{j=1}^{r}c_{ij}\delta_{j},$ for all $1\leq i\leq k.$ This contradicts the fact that
   	$\tilde{c}_s\notin span~ \big\{\tilde{c}_1,\tilde{c}_2,\dots,\tilde{c}_r \big\}.$ So our claim is established.
   	Therefore, for this $(\gamma_1,\gamma_2,\dots,\gamma_r)=-(\frac{c_{q1}}{c_{qs}},\frac{c_{q2}}{c_{qs}},\dots,\frac{c_{qr}}{c_{qs}}),$ the equation (\ref{eq1}) has a solution in $\mathbb{R}^{r+1}.$ Let the solution be $(\tilde{\eta_1},\tilde{\eta_2},\dots,\tilde{\eta_r},\sum_{j=1}^{r}\gamma_j\tilde{\eta_j }).$\\ 
   	Let us consider $\eta$ as $[ \eta]_{\mathcal{B}_1} =(\eta_1,\eta_2,\dots,\eta_n)^t,$ where
   	\begin{eqnarray*}
   		\eta_i&=&\tilde{\eta_i}, \quad\quad\,\,\text{ if }1\leq i\leq r\\
   		&=&\sum_{j=1}^{r}\gamma_j\tilde{\eta_j}, \text{ if }i=s\\
   		&=&0, \quad\quad\,\,\,\,\text{ otherwise.}
   	\end{eqnarray*}
   	It is easy to see that $\eta\in ker(\tilde{f_q})$. Hence $\eta \in A$. Then for each $1\leq i\leq r,$ $g_{h_i}(T\eta)=\sum_{j=1}^{n}d_{h_ij}\eta_j=
   	\sum_{j=1}^{r}d_{h_ij}\eta_j+d_{h_is}\eta_s=\sum_{j=1}^{r}d''_{ij}\tilde\eta_j+d''_{i r+1}\sum_{j=1}^{r}\gamma_j\tilde{\eta_j}=g_{h_i}(Tz).$ 
   	So for $i\in\{h_1,h_2,\dots,h_r\},$ $ g_i(T\eta)=g_i(Tz).$
   	Now, if $r<p$ then for  $i\in \{1,2,\dots,p\}\backslash\{h_1,h_2,\dots,h_r\},$
   	\begin{eqnarray*}
   		g_i(T\eta)&=& \sum_{t=1}^{r}\lambda_{ih_t}g_{h_t}(T\eta)\\
   		&=& \sum_{t=1}^{r}\lambda_{ih_t}g_{h_t}(Tz)\\
   		&=&  g_i(Tz).
   	\end{eqnarray*}
   	Hence $g_i(T\eta)= g_i(Tz)$ for all $1\leq i\leq p.$ This implies $Tz\in B.$ Since $z$ is chosen arbitrarily from $\mathbb X,$ it follows that $T(\mathbb X)\subset B.$
   \end{proof}

\begin{remark}
	The previous result can be interpreted geometrically in a visually appealing manner. Indeed, it is not difficult to observe that the above theorem admits the following reformulation: Given any $ 1 \leq r < k, $ and any linear map $ T $ on $ \mathbb{R}^{n}, $ if the union of the kernels of all possible convex combinations of $ k $ linearly independent functionals is mapped by $ T $ into the union of the kernels of all possible convex combinations of $ k-r $ linearly independent functionals, then $ T $ necessarily maps $ \mathbb{R}^{n} $ to the kernel of some particular convex combination of the latter $  k-r $ functionals. Since the kernel of any non-zero functional is nothing but a maximal hyperplane in $ \mathbb{R}^n, $ the above theorem asserts that the image of a linear map is contained in a (maximal) hyperplane, provided it maps certain hyperplanes to certain hyperplanes. In particular, for a bijective linear map on $ \mathbb{R}^{n}, $ such a possibility does not exist. 
\end{remark}

As we will see in the next result, the previous property of linear maps turns out to be the key ingredient in producing a connection between local preservation of Birkhoff-James orthogonality by a linear operator and the theory of $ k $-smoothness.

   \begin{theorem}\label{pokm}
   	Let $ \mathbb{X} $ and $ \mathbb{Y} $ be Banach spaces of dimension $n$ and $m,$ respectively. Let $T\in \mathbb L(\mathbb{X},\mathbb{Y})$ and let $ x \in \mathbb{X}$ be a $k$-smooth point of $ \mathbb{X} $ such that $ T $ preserves Birkhoff-James orthogonality at $ x. $ Then  either $Tx = 0 $  or $Tx$  is a $p$-smooth point of $\mathbb{Y},$  for some $k\leq p.$
   \end{theorem}

   \begin{proof}
   	Let $x\in \mathbb{X}$ be a $k$-smooth point. Then there exist $k$ linearly independent $f_1, f_2, ..., f_k\in J(x).$ Consider the set $F=co(\{f_1, f_2, ..., f_k\})$. Suppose that $A=\bigcup\limits_{f\in F}ker(f).$ Clearly, $A\subset x^{\perp_B}.$ Let $Tx\neq 0.$ If possible, suppose that $Tx$ is a $p$-smooth point, for some $p<k.$ Then there exist $p$ linearly independent $g_1, g_2, ..., g_p\in J(Tx).$ Consider the set  $G=\Big \{\sum_{i=1}^{p}\beta_if_i :\sum_{i=1}^{p}\beta_i=1, \beta_i\in \mathbb{R} \,\, \forall\, i =1,2,\ldots,p\Big \}.$ Suppose that $B=\bigcup\limits_{g\in G}ker(g).$  
   	We claim that $J(Tx)\subset G.$ Let $g\in J(Tx),$ then $g$ can be expressed as $g=\sum_{i=1}^{p}c_ig_i,$ where $c_i\in\mathbb{R}.$ 
   	Since $g,g_1, g_2, ..., g_p\in J(Tx)$ , it follows that $\sum_{i=1}^{p}c_i=1.$ Hence $g\in G.$ Thus our claim is established.
   	Observe that $ J(Tx)\subset G  \implies {(Tx)}^{\perp_B}\subset B.$ By the hypothesis,  $T(x^{\perp_B}) \subset {(Tx)}^{\perp_B}.$ Thus we get,  $T(A)\subset B.$ Now, it follows from Theorem \ref{th1} that $T(\mathbb{X})\subset B.$ But for each $g\in G$, $g(Tx)=\| Tx\|$  and so  $Tx\notin B.$ This contradicts the fact that $T(\mathbb{X})\subset B.$ Therefore, our assumption that $Tx$ is a $p$-smooth point for some $p<k$ is not correct. Hence  $Tx$ is a $p$-smooth point for some $p\geq k.$ This completes the proof.
   \end{proof}

The following corollary to the above theorem is particularly useful in obtaining refinements of the Blanco-Koldobsky-Turnsek Theorem in certain specific cases. We omit the proof, since in light of the previous result, it is clear from the well-known fact that in any polyhedral Banach space of dimension $ n, $ a unit vector is an extreme point of the unit ball if and only if it is $ n $-smooth.

   \begin{cor}\label{ioe}
   	Let $ \mathbb{X} $ and $ \mathbb{Y} $ be $n$-dimensional polyhedral Banach spaces.  If $T\in\mathbb{L(\mathbb{X},\mathbb{Y})}$ preserves Birkhoff-James orthogonality at an extreme point $x$ of $B_{\mathbb X},$ then $Tx$ is a scalar multiple of some extreme point of $B_{\mathbb Y}.$
   \end{cor}

In case we consider the preservation of Birkhoff-James orthogonality in the opposite direction (the meaning of which will be clear from the statement of the next theorem), then we have the following information on the preservation of order of smoothness.

   \begin{theorem}\label{poorp}
   	Let $ \mathbb{X} $ and $ \mathbb{Y} $ be Banach spaces of dimensions $n$ and $m,$ respectively. Suppose that $x$ is a $k$-smooth point of $\mathbb{X}.$  If $T\in \mathbb L(\mathbb{X},\mathbb{Y})$ is such that for any $y\in \mathbb{X},$ $Tx\perp_{B}Ty\implies x\perp_{B}y,$ then $Tx$ is a $p$-smooth point of $Range(T),$ for some $p\leq k.$
   \end{theorem}
   \begin{proof}
   	Let $T\in \mathbb L(\mathbb{X},\mathbb{Y})$ be such that for any $y\in \mathbb{X}$, $Tx\perp_{B}Ty\implies x\perp_{B}y.$ Let the rank of $T$ be $r$. Clearly, $x \neq 0$ and by the hypothesis $ Tx \neq 0$ so that $ x \not \in ker(T).$   Let $\{u_1,u_2,\dots,u_{n-r}\}$ be a basis of $ ker(T).$ Then 
   	$\{u_1,u_2,\dots,u_{n-r},x\}$ is a linearly independent set of $\mathbb{X}$ 
   	and  can be extended to a basis of $\mathbb{X}$ by adding the elements $v_1,v_2,\dots,v_{r-1}.$ Let $\mathbb{V}=span~\{v_1,v_2,\dots,v_{r-1},x\}.$ Let $T_{\mathbb{V}}$ be the restriction of $T$ on $\mathbb{V},$ i.e., $T_{\mathbb{V}}:\mathbb{V}\rightarrow Range(T)$ be such that $T_{\mathbb{V}}z=Tz$ for all $z\in\mathbb{V}.$ Clearly, $T_{\mathbb{V}}$ is bijective and so $T_{\mathbb{V}}^{-1}:Range(T)\rightarrow \mathbb{V}$ exists.
   	Now, $T_{\mathbb{V}}x\perp_{B}T_{\mathbb{V}}y\implies x\perp_{B}y,$ i.e., $T_{\mathbb{V}}x\perp_{B}T_{\mathbb{V}}y\implies T_{\mathbb{V}}^{-1}(T_{\mathbb{V}}x)\perp_{B}T_{\mathbb{V}}^{-1}(T_{\mathbb{V}}y).$ Therefore, it follows from Theorem \ref{pokm} that if $T_{\mathbb{V}}x$ is a $p$-smooth point of $Range(T)$ then $T_{\mathbb{V}}^{-1}(T_{\mathbb{V}}x)$ is a $p_1$-smooth point of $\mathbb{V}$ for some $p_1\geq p.$ So $x$ is $p_1$-smooth
   	point of $\mathbb{V},$ where $p_1\geq p.$ Since  $x$ is $k$-smooth
   	point of $\mathbb{X},$ it follows that $k\geq p_1.$ Thus $ k \geq p.$ This completes the proof.
   \end{proof}
   Combining Theorem \ref{pokm} with Theorem \ref{poorp}, we get the following result.
   \begin{theorem}
   	Let $ \mathbb{X} $ and $ \mathbb{Y} $ be Banach spaces of dimension $n$ and $m,$ respectively. Suppose that $x$ is a $k$-smooth point of $\mathbb{X}.$ If $T\in \mathbb L(\mathbb{X},\mathbb{Y})$ be such that for any $y\in \mathbb{X},$ $Tx\perp_{B}Ty\iff x\perp_{B}y,$ then $Tx$ is a $k$-smooth point of $Range(T).$
   \end{theorem}

In the remaining part of the article, our aim is to study the role of local preservation of Birkhoff-James orthogonality in understanding isometries on polyhedral Banach spaces, including some important specific examples. We begin by considering the two-dimensional case.

    \begin{lemma}\label{bijective}
   	Let $ \mathbb{X} $ and $\mathbb{Y} $ be two-dimensional polyhedral Banach spaces. If a non-zero operator $T\in\mathbb L(\mathbb{X},\mathbb Y)$ preserves Birkhoff-James orthogonality at any two linearly independent extreme points of $B_\mathbb{X},$ then $T$ is bijective.
   \end{lemma}

   \begin{proof}
   	Let $T\in\mathbb L(\mathbb{X},\mathbb Y)$ be such that $T$ preserves Birkhoff-James orthogonality at two linearly independent extreme points $x_1,x_2$ of $B_\mathbb{X}.$	Since $T\neq 0,$ it follows that either $Tx_1\neq0$ or $Tx_2\neq0.$ Without loss of generality assume that $Tx_1=y\neq 0.$ If possible, suppose that $T$ is not bijective. So for every $z\in\mathbb{X},$ $Tz=ky$ for some $k\in\mathbb{R}.$ Hence $y^{\perp_B}\cap\{Tz:z\in\mathbb{X}\}=\{0\}.$ Since $\mathbb X$ is polyhedral and $x_1$ is an extreme point of $B_\mathbb{X},$ there exist two linearly independent $z_1,z_2\in\mathbb{X}$ such that $x_1\perp_{B}z_1$ and  $x_1\perp_{B}z_2.$ Hence $Tz_1=Tz_2=0.$ This contradicts the hypothesis that  $T\neq 0.$ Therefore, $T$ is bijective.
   \end{proof}

The above lemma helps us to study the preservation of extreme points (up to a suitable scaling) under the assumption of preservation of Birkhoff-James orthogonality at each extreme point of the unit ball of the domain space.

   \begin{theorem}\label{icext}
   	Let $ \mathbb{X} $ and $\mathbb{Y} $ be two-dimensional polyhedral Banach spaces whose unit spheres are $2n$-gon for some $n\in\mathbb N.$ If a non-zero operator $T\in \mathbb L(\mathbb{X},\mathbb{Y})$ preserves  Birkhoff-James orthogonality at each extreme point of $B_\mathbb{X}, $ then $T$ maps two consecutive extreme points of $B_\mathbb{X} $ to scalar multiples of two consecutive extreme points of $B_\mathbb{Y}.$ 
   \end{theorem}

   \begin{proof}
   	Let $T\in \mathbb L(\mathbb X,\mathbb Y)$ be such that $T$ preserves Birkhoff-James orthogonality at each extreme point of $B_\mathbb{X} .$ Therefore, it follows from Corollary \ref{ioe} that the image of each extreme point of $B_\mathbb{X} $ is a scalar multiple of some extreme point of $B_\mathbb{Y} .$ Applying Lemma \ref{bijective}, we conclude that $T$ is bijective.
   	If possible suppose that $Tu_1=k_1v_k$ and $Tu_2=k_2v_l,$ where $u_1, u_2$ are consecutive extreme points of $B_\mathbb{X} $ but $v_k,v_l$ are two non-consecutive extreme points of $B_\mathbb{Y} $ and $k_1,k_2>0.$ Since $v_k, v_l$ are two non consecutive extreme points of $B_\mathbb{Y},$ there exists at least one extreme point $v$ of $B_\mathbb{Y} $ such that $v=\alpha (1-t)v_k+\alpha  t v_l$ for some $\alpha>0$ and $t\in(0,1).$ Since $T$ is bijective and the image of  each extreme point of $B_\mathbb{X} $ is a scalar multiple of some extreme point of $B_\mathbb{Y},$ it follows that inverse image of  each extreme point of $B_\mathbb{Y} $ is also a scalar multiple of some extreme point of $B_\mathbb{X}.$ So $T^{-1}v$ is a scalar multiple of an extreme point of $B_\mathbb{X}.$  Now, $T^{-1}v=T^{-1}(\alpha (1-t)v_k+\alpha  t v_l)=(1-t)\frac{\alpha}{k_1}u_1+t \frac{\alpha}{k_2} u_2.$ This contradicts the hypothesis that $u_1,u_2$ are two consecutive extreme points of $B_\mathbb{X}.$ Hence the images of two consecutive extreme points of $B_\mathbb{X}$ are scalar multiples of two consecutive extreme points of $B_\mathbb{Y}.$
   \end{proof}

Our next result gives a partial refinement of the Blanco-Koldobsky-Turnsek Theorem in case of two-dimensional polyhedral Banach spaces, under an additional assumption on the norm of the images of the extreme points of the unit ball of the domain.

   \begin{theorem}
   	Let $ \mathbb{X} $ be a two-dimensional polyhedral Banach space whose unit sphere is a $2n$-gon for some $n\in\mathbb N.$ Let $T\in \mathbb L(\mathbb{X}).$ Then the following are equivalent: 
   	\begin{itemize}
   		\item[(i)] $T$ is a scalar multiple of an isometry.
   		\item[(ii)]$T$ preserves Birkhoff-James orthogonality at each extreme point of $B_\mathbb{X}.$ Moreover, for each pair of consecutive extreme points $x_1,x_2$ of  $B_\mathbb{X} $ there exist $s,t\in[0,1]$ such that $u=tx_1+(1-t)x_2,v=sx_1+(1-s)x_2$ and $\|Tu\|=\|Tv\|.$ 
   		\item[(iii)] $T$ preserves Birkhoff-James orthogonality at each extreme point of $B_\mathbb{X}.$ Moreover, $\| Tu\|=\|Tv\|$ for any extreme points $u,v$ of  $B_\mathbb{X}.$
   	\end{itemize} 
   \end{theorem}
   \begin{proof}
   (i)$\implies$(ii): Follows immediately. \\ 
   
   (ii)$\implies$(iii): If $T=0$ then it follows trivially. Next assume that $T\neq0.$ Let $x_1,x_2$ be two  consecutive extreme points of  $B_\mathbb{X} $ and let $u=tx_1+(1-t)x_2,v=sx_1+(1-s)x_2,$ for some  $s,t\in[0,1],$ such that $\| Tu\|=\|Tv\|=k.$ Clearly, $k\neq0.$  Let $S=\frac{T}{k}.$ Since Birkhoff-James orthogonality is homogeneous, it follows that $S$ preserves Birkhoff-James orthogonality at each extreme point of $B_\mathbb{X}.$ It follows from Theorem \ref{icext} that $y_1$ and $y_2$ are two consecutive extreme points of  $B_\mathbb{X}$, where  $Sx_1=k_1y_1$ and $Sx_2=k_2y_2$ for some  $k_1,k_2>0.$ We claim that $Tu$ and $Tv$ must belong to the same edge of $S_\mathbb{X}.$ On the contrary to our claim suppose that  $Tu$ and $Tv$  belong to two different edges of $S_\mathbb{X}.$ Then there exists an extreme point $y$ between  $Su$ and $Sv$ of $B_\mathbb{X}$ such that $y=\alpha(1-t)Su+\alpha t Sv$ for some $t\in(0,1),\alpha>0.$ Since  $u,v$ belong to the line segment joining $x_1$ and $x_2,$ it follows that $Su,Sv$ belong to the line segment joining $Sx_1$ and $Sx_2.$ So $y=\alpha(1-\lambda)Sx_1+\alpha \lambda Sx_2$ for some $\lambda\in (0,1),$ i.e., $y=\alpha(1-\lambda)k_1y_1+\alpha \lambda k_2y_2.$ This contradicts the fact that $y_1,y_2$ are two consecutive extreme points of $B_\mathbb{X}.$ Thus our claim is established. Let $H$ be a hyperplane supporting to the edge containing $Su$ and $Sv$. So $Sx_1,Su,Sv,Sx_2$ belong to $H.$ Next, we claim that $\|Sx_1\|=\|Sx_2\|.$ On the contrary suppose that $\|Sx_1\|\neq\|Sx_2\|.$ Without loss of generality assume that $\|Sx_1\|>\|Sx_2\|.$ Then $\|Su\|=\|S(tx_1+(1-t)x_2)\|\leq t\|Sx_1\|+(1-t)\|Sx_2\|<\|Sx_1\|.$ This implies that there exists an extreme point $z$ of $B_\mathbb{X}$ between  $Su$ and $Sx_1.$ This contradicts the fact that $y_1$ and $y_2$ are consecutive extreme points of $B_\mathbb{X}.$ Thus, $\|Sx_1\|=\|Sx_2\|.$ This implies  $\|Tx_1\|=\|Tx_2\|.$ Therefore, $\| Tv_1\|=\|Tv_2\|$ for any extreme points $v_1,v_2$ of  $B_\mathbb{X}.$ \\
   
   (iii)$\implies$(i): Let $T\in \mathbb L(\mathbb{X})$ be such that $T$  preserves Birkhoff-James orthogonality at each extreme point of $B_\mathbb{X} $ and  $\| Tu\|=\|Tv\|$ for any two extreme points $u,v$ of  $B_\mathbb{X}.$ If $T=0$ then (i) follows trivially. Let $T\neq0$, then applying Lemma \ref{bijective}, we conclude that $T$ is bijective. From Corollary \ref{ioe}, it follows that the image of each extreme point of $B_\mathbb{X} $ is a scalar multiple of some extreme point of $B_\mathbb{X}.$ Since $\mathbb{X}$ is finite dimensional, $T$ must attains its norm at an  extreme point of $B_\mathbb{X}.$ Also,  we have  $\| Tu\|=\|Tv\|$ for any two extreme points $u,v$ of  $B_\mathbb{X}.$  Thus  $\| Tu\|=\|T\|$ for all extreme points $u$ of $B_\mathbb{X}.$ Let $S=\frac{T}{\|T\|}.$ Then $S$ maps each extreme point of $B_\mathbb{X} $ to some extreme point of $B_\mathbb{X}.$ Since $T$ is bijective, $S$ is bijective. Hence $S^{-1}$ exists and maps each extreme point of $B_\mathbb{X} $ to some extreme point of $B_\mathbb{X}.$ Now, $S^{-1}$ also attains its norm at some extreme point of $B_\mathbb{X}.$ So $\|S^{-1}\|=1.$ 
   	Hence  $\|Sx\|\leq \|x\|$ for all $x\in\mathbb{X}$ and $\|x\|=\|S^{-1}(Sx)\|\leq \|Sx\|.$ This implies $\|Sx\|=\|x\|$ for all  $x\in\mathbb{X}.$ So $S$ is an isometry. Thus $T$ is a scalar multiple of an isometry. 
   \end{proof}

The previous theorem requires an extra condition, in addition to the assumption of preservation of Birkhoff-James orthogonality at each extreme point of the unit ball of the domain. However, as we will prove in the next theorem, a proper refinement of the Blanco-Koldobsky-Turnsek Theorem is indeed possible in case of a certain family of two-dimensional polyhedral Banach spaces.

 \begin{theorem}
	Let $\mathbb{X}$ be the two-dimensional polyhedral Banach space whose unit sphere is the  regular $2n$-gon with vertices $$v_j=(\cos(j-1)\frac{\pi}{n},\sin(j-1)\frac{\pi}{n}), 1\leq j\leq 2n.$$ If  $T\in\mathbb L(\mathbb X)$ preserves  Birkhoff-James orthogonality at any two consecutive vertices of $S_{\mathbb X}$ then $T$ is a scalar multiple of an isometry.
\end{theorem}
\begin{proof}
	If $T=0$ then the theorem is trivially true. Let $T\neq0.$ First we claim that if $T$ preserves Birkhoff-James orthogonality at $v_i,v_{i+1}$ ($v_{2n+1}=v_1$)  for some $1\leq i\leq  2n$ then $\|Tv_i\|=\|Tv_{i+1}\|.$ It follows from Lemma \ref{bijective} that $T$ is bijective. From Corollary \ref{ioe}, we have $Tv_i=k_1v_k$ and $Tv_{i+1}=k_2v_l$ for some $1\leq k,l\leq 2n$ and $k_1,k_2(\neq0)\in\mathbb{R}.$ Note that the support functional supporting the edge of $S_\mathbb X$ joining $v_i$ and $v_{i+1}$ is $$f(\alpha,\beta )=\frac{1}{\cos\frac{\pi}{2n}}\Big(\alpha\cos\frac{(2i-1)\pi}{2n}+ \beta\sin\frac{(2i-1)\pi}{2n}\Big).$$ 
	Now, it is easy to see that $f({v_i-v_{i+1}})=0.$ So  $v_i \perp_B ({v_i-v_{i+1}}) $ and $     v_{i+1}\perp_B (v_i-v_{i+1}).$ Since $T$ preserves Birkhoff-James orthogonality at $v_i,v_{i+1}$ it follows that $Tv_i\perp_B T({v_i-v_{i+1}}),    Tv_{i+1}\perp_B T({v_i-v_{i+1}})$ and this implies that $v_k\perp_B {(k_1v_k-k_2v_l)},   v_l\perp_B {(k_1v_k-k_2v_l)}.$ Then $  \|v_k+\lambda {(k_1v_k-k_2v_l)}\|\geq \|v_1\|$ for all $\lambda\in\mathbb{R}.$ Let $\lambda=-\frac{1}{k_1},$ then $\frac{|k_2| }{|k_1|}\geq1.$ Similarly, from $v_l\perp_{B}{(k_1v_k-k_2v_l)},$ we get  $\frac{|k_1|}{|k_2|}\geq1.$
	So $|k_1|=|k_2|.$ Therefore, $\|Tv_i\|=\|Tv_{i+1}\|.$\\
	Without loss of generality assume that   $T$ preserves Birkhoff-James orthogonality at $v_1,v_2.$ So  $Tv_1=kv_s$ and $Tv_2=kv_t$  for some $1\leq s,t\leq 2n$ and $k(\neq0)\in\mathbb{R}.$ Let $T_1=\frac{T}{k}.$ Hence
	\begin{eqnarray*}
		T_1(x,y)&=&\frac{x\sin\frac{\pi}{n}-y\cos\frac{\pi}{n}}{\sin\frac{\pi}{n}}\Big(\cos\frac{(s-1)\pi}{n},\sin\frac{(s-1)\pi}{n}\Big)\\
		&~~&+\frac{y}{\sin\frac{\pi}{n}}\Big(\cos\frac{(t-1)\pi}{n},\sin\frac{(t-1)\pi}{n}\Big)\\
		&=&\frac{1}{\sin\frac{\pi}{n}}\Big(x\sin\frac{\pi}{n}\cos\frac{(s-1)\pi}{n}-y\big(\cos\frac{\pi}{n}\cos\frac{(s-1)\pi}{n}-\cos\frac{(t-1)\pi}{n}\big),\\
		&~~&x\sin\frac{\pi}{n}\sin\frac{(s-1)\pi}{n}-y\big(\cos\frac{\pi}{n}\sin\frac{(s-1)\pi}{n}-\sin\frac{(t-1)\pi}{n}\big)\Big).
	\end{eqnarray*}
	Since $v_1\perp_B (0,1)$, $v_s\perp_B T_1(0,1).$ Now $$T_1(0,1)=\frac{1}{\sin\frac{\pi}{n}}\Big(\cos\frac{(t-1)\pi}{n}-\cos\frac{\pi}{n}\cos\frac{(s-1)\pi}{n},\sin\frac{(t-1)\pi}{n}-\cos\frac{\pi}{n}\sin\frac{(s-1)\pi}{n}\Big).$$
	Let $f_1$ and $f_2$ be two extreme support functionals of $v_s.$ Then
	\begin{eqnarray*}
		f_1(\alpha,\beta ) &=&\frac{1}{\cos\frac{\pi}{2n}}\Big(\alpha\cos\frac{(2s-3)\pi}{2n}+\beta\sin\frac{(2s-3)\pi}{2n}\Big)\\
		\mbox{and}\,f_2(\alpha,\beta )&=&\frac{1}{\cos\frac{\pi}{2n}}\Big(\alpha\cos\frac{(2s-1)\pi}{2n}+\beta\sin\frac{(2s-1)\pi}{2n}\Big).
	\end{eqnarray*}
	Hence 
	\begin{eqnarray*}
		f_1(T_1(0,1))&=&\frac{1}{\cos\frac{\pi}{2n}\sin\frac{\pi}{n}}\Big(\cos\frac{(2(t-s)+1)}{2n}-\cos\frac{\pi}{2n}\cos\frac{\pi}{n}\Big)\\
		\mbox{and}\,f_2(T_1(0,1))&=&\frac{1}{\cos\frac{\pi}{2n}\sin\frac{\pi}{n}}\Big(\cos\frac{(2(t-s)-1)}{2n}-\cos\frac{\pi}{2n}\cos\frac{\pi}{n}\Big).
	\end{eqnarray*}
	Since $v_s\perp_B T_1(0,1),$ there exists a   $\lambda\in[0,1]$ such that $$(1-\lambda)f_1(T_1(0,1))+\lambda f_2(T_1(0,1))=0.$$
	This implies $f_1(T_1(0,1))f_2(T_1(0,1))\leq0.$ Since $|t-s|<2n,$ it follows from $f_1(T_1(0,1))f_2(T_1(0,1))\leq0$ that $|t-s|=1,2n-1.$
	Therefore, either $T_1v_1=v_s $ and $  T_1v_2=v_{s+1}$ or $T_1v_1=v_{s+1} $ and $  T_1v_2=v_s$ for some  $1\leq s\leq 2n.$
	If $T_1v_1=v_s $ and $  T_1v_2=v_{s+1},$ then 
	\begin{eqnarray*}
		T_1(x,y)
		&=&\frac{x\sin\frac{\pi}{n}-y\cos\frac{\pi}{n}}{\sin\frac{\pi}{n}}\Big(\cos\frac{(s-1)\pi}{n},\sin\frac{(s-1)\pi}{n}\Big)+\frac{y}{\sin\frac{\pi}{n}}\Big(\cos\frac{s\pi}{n},\sin\frac{s\pi}{n}\Big)\\
		&=&\Big(x\cos\frac{(s-1)\pi}{n}-y\sin\frac{(s-1)\pi}{n},x\sin\frac{(s-1)\pi}{n}+y\cos\frac{(s-1)\pi}{n}\Big)\\
		&=&\begin{pmatrix}
			\cos\frac{(s-1)\pi}{n}&-\sin\frac{(s-1)\pi}{n}\\
			\sin\frac{(s-1)\pi}{n}&\cos\frac{(s-1)\pi}{n}
		\end{pmatrix}
		\begin{pmatrix}
			x\\
			y
		\end{pmatrix}.
	\end{eqnarray*}
	This implies that $T_1$ is an isometry.
	If $T_1v_1=v_{s+1} $ and $ T_1v_2=v_s,$ then 
	\begin{eqnarray*}
		T_1(x,y)
		&=&\frac{x\sin\frac{\pi}{n}-y\cos\frac{\pi}{n}}{\sin\frac{\pi}{n}}\Big(\cos\frac{s\pi}{n},\sin\frac{s\pi}{n}\Big)+\frac{y}{\sin\frac{\pi}{n}}\Big(\cos\frac{(s-1)\pi}{n},\sin\frac{(s-1)\pi}{n}\Big)\\
		&=&\Big(x\cos\frac{s\pi}{n}+y\sin\frac{s\pi}{n},x\sin\frac{s\pi}{n}-y\cos\frac{s\pi}{n}\Big)\\
		&=&\begin{pmatrix}
			\cos\frac{s\pi}{n}&\sin\frac{s\pi}{n}\\
			\sin\frac{s\pi}{n}&-\cos\frac{s\pi}{n}
		\end{pmatrix}
		\begin{pmatrix}
			x\\
			y
		\end{pmatrix}.
	\end{eqnarray*}
	Again this implies that  $T_1$ is an isometry. This completes the proof of the theorem.
	
\end{proof}

\begin{remark}
	It is clear from the above theorem that if $\mathbb{X}$ is the  two-dimensional polyhedral Banach space whose unit sphere is the  regular $2n$-gon then $\kappa(\mathbb {X})=2.$
	\end{remark}

In case of a general polyhedral Banach space of dimension strictly greater than $ 2, $ obtaining a refinement of the Blanco-Koldobsky-Turnsek Theorem seems to be a difficult task. This is primarily due to the lack of regularity in the geometric structures of the unit balls of such spaces. Nevertheless, as we will see next in the specific cases of $ \ell_{\infty}^n $ and $ \ell_{1}^n, $ substantial refinements of the said theorem is indeed possible when the norm is well-behaved.

   \begin{lemma}\label{nel}
   	Let $\mathbb{X}=\ell_{\infty}^n.$ If $T\in \mathbb L(\mathbb{X})$ preserves Birkhoff-James orthogonality at two extreme points $u_1,u_2$ of $B_{\mathbb{X}}$, then $\|Tu_1\|=\|Tu_2\|.$
   \end{lemma}
   \begin{proof}
   	Let $T\in\mathbb L(\mathbb{X})$ be such that $T$ preserves Birkhoff-James orthogonality at two extreme points $u_1,u_2$  of $B_{\mathbb{X}}.$ Therefore, it follows from Corollary \ref{ioe} that $Tu_1=k_1v_1$ and $Tu_2=k_2v_2$ for some extreme points $v_1,v_2$ and $k_1,k_2\in\mathbb{R}.$
   	If $u_1,u_2$ are linearly dependent, then $u_1=\alpha u_2$ for some $\alpha\in\mathbb R.$ Since $u_1,u_2$ are extreme points of $B_{\mathbb{X}},$ it follows that $|\alpha|=1$. Now, $|k_1|=\|k_1v_1\|=\|Tu_1\|=\|T(\alpha u_2)\|=\|\alpha k_2 v_2\|=|k_2|=\|Tu_2\|.$ Hence $\|Tu_1\|=\|Tu_2\|.$ 
   	Next, assume that $u_1,u_2$ are linearly independent. We claim that $u_1,u_2\perp_{B} {(u_1+u_2)}.$  Let $u_1=(a_1,a_2, \ldots, a_n) $ and $ u_2 = (b_1,b_2,\ldots,b_n),$ where $ a_i=\pm 1, b_j = \pm 1,1\leq i,j\leq n.$ Since $u_1,u_2$ are linearly independent, there exist $r,s \in\{1,2,\dots,n\}$ such that
   	$a_r= b_r $ and $ a_s = -b_s.$  Then  $ u_1 + u_2  = ( a_1+b_1, a_2 + b_2, \ldots, a_r + b_r, \ldots, a_s + b_s, \ldots, a_n + b_n)= (a_1 +b_1, a_2+b_2, \ldots, 2a_r,\ldots,0,\ldots,a_n+b_n).$  It is easy to verify that $u_1,u_2\perp_{B} {(u_1+u_2)}.$ Thus our claim is established. 
   	If $|k_1|=|k_2|=0$ then $\|Tu_1\|=\|Tu_2\|.$
   	If  $|k_1|=0$, then $u_2\perp_{B}{(u_1+u_2)}\implies Tu_2\perp_{B}T{(u_1+u_2)}\implies k_2v_2\perp_{B}{(k_2v_2)}\implies k_2=0.$ So $\|Tu_1\|=\|Tu_2\|=0.$
   	Similarly, if  $|k_2|=0,$ then $\|Tu_1\|=\|Tu_2\|=0.$
   	Finally, assume that  $|k_1|\neq 0$ and  $|k_2|\neq 0.$ Since $T$ preserves Birkhoff-James orthogonality at $u_1,u_2,$ it follows that $v_1\perp_{B}{(k_1v_1+k_2v_2)}$ and $v_2\perp_{B}{(k_1v_1+k_2v_2)}.$
   	Now, $v_1\perp_{B}{(k_1v_1+k_2v_2)}\implies \|v_1+\lambda {(k_1v_1+k_2v_2)}\|\geq \|v_1\|$ for all $\lambda\in\mathbb{R}.$ Let $\lambda=\frac{-1}{k_1},$ then $\frac{|k_2| }{|k_1|}\geq1.$ Similarly, from $v_2\perp_{B}{(k_1v_1+k_2v_2)} ,$ we have  $\frac{|k_1|}{|k_2|}\geq1.$
   	Therefore, $|k_2|=|k_1|.$ Hence $\|Tu_1\|=\|Tu_2\|.$
   \end{proof}
   
   We next present a refinement of the Blanco-Koldobsky-Turnsek Theorem for $ \ell_{\infty}^n. $
   
   \begin{theorem}
   	Let $\mathbb{X}=\ell_{\infty}^n.$ If $T\in \mathbb L(\mathbb{X})$ preserves Birkhoff-James orthogonality at all extreme points of $B_{\mathbb{X}},$ then $T$ is a scalar multiple of an isometry.
   \end{theorem}
   \begin{proof}
   	We complete the proof in three steps. In Step 1, we show that there exists an Auerbach basis for $\mathbb X.$ Next in Step 2, we show that $T$ is either the zero operator or bijective. Finally in Step 3,  we show that $T$ is a  scalar multiple of an isometry.
   	
   	\medskip
   	
   	\noindent \textbf{Step 1:}
   	Let $A=\{e_i=(e_i(1),e_i(2),\dots,e_i(n)):1\leq i\leq n\}\subset Ext(B_{\mathbb{X}}),$ where for each $1 \leq i,j\leq n,$
   	\begin{eqnarray*}
   		e_i(j)&=-1, & j<i\\
   		&=1, & j\geq i .
   	\end{eqnarray*} 
   	Clearly, $A$ is a basis for $\mathbb X.$ We claim that $A$ is an Auerbach basis for $\mathbb X.$  For this first we show that $e_1\perp_B span~(A\backslash \{e_1\}).$ Let $x\in span~(A\backslash \{e_1\}) $ be arbitrary, then $x=(a_1,a_2,\dots,a_n)=\sum_{j=2}^{n}\alpha_je_j,$ for some $\alpha_j\in\mathbb R.$ Then it is easy to see that $a_n=-a_1=\alpha_2+\alpha_3+\dots+\alpha_n$. So $\|e_1+\lambda x\|\geq max\{|1-\lambda a_1|,|1+\lambda a_1|\}\geq1=\|e_1\|.$ Hence $e_1\perp_B x.$ Since $x$ is taken arbitrarily from $span~(A\backslash\{e_1\})$, it follows that  $e_1\perp_B span~(A\backslash \{e_1\}).$ Next, we  show that for each $2\leq i\leq n$,  $e_i \perp_B span~(A \setminus \{e_i\}).$  Let $x\in span~(A\backslash\{e_i\}) $ be arbitrary, then  $x=(a_1,a_2,\dots,a_n)=\sum_{j=1,j\neq i}^{n}\alpha_je_j,$ for some $\alpha_j\in\mathbb R.$  Then it is easy to see that $a_{i-1}=a_i=\alpha_1+\alpha_2+\dots+\alpha_{i-1}-\alpha_{i+1}-\dots-\alpha_n$. So $\|e_i+\lambda x\| \geq max\{|1+\lambda a_i|,|-1+\lambda a_i|\}\geq1=\|e_i\|.$ Hence $e_i \perp_B x.$ Since $x$ is taken arbitrarily from $span~(A \setminus \{e_i\}),$ it follows that  $e_i \perp_B span~(A \setminus \{e_i\})$ for each $2\leq i\leq n.$  Thus our claim is established that $A$ is an Auerbach basis for $\mathbb X.$
   	
   	\medskip
   	
   	\noindent \textbf{Step 2:} Let $T\neq0.$ We show that  $T$ is bijective. On the contrary suppose that $T$ is not bijective. So there exists $ i\in\{1,2,\dots,n\}$ such that $Te_i\in span~T(A \setminus \{e_i\}).$ Then $Te_i=\sum_{j=1,j\neq i}^{n}\beta_jTe_j$ for some $\beta_j\in\mathbb R .$  This implies that $Te_i=T \big (\sum_{j=1,j\neq i}^{n}\beta_je_j \big ).$ From Step 1 we have $e_i\perp_B \sum_{j=1,j\neq i}^{n}\beta_je_j.$ Since $T$ preserves Birkhoff-James orthogonality at $e_i$, it follows that $Te_i\perp_B T\big (\sum_{j=1,j\neq i}^{n}\beta_je_j \big).$ This implies that $Te_i=0.$ Applying Lemma \ref{nel}, we conclude that $Te_k=0$ for all $k=1,2,\dots,n.$ This  contradicts the hypothesis that $T\neq0.$ Therefore, T is bijective.
   	
   	\medskip
   	\noindent \textbf{Step 3:} From Corollary \ref{ioe} it follows that image of each extreme point of $B_\mathbb{X} $ is a scalar multiple of some extreme point of $B_\mathbb{X} .$ Since $\mathbb{X}$ is finite dimensional, $T$ attains its norm at an  extreme point of $B_\mathbb{X}.$ Since $\| Tu\|=\|Tv\|$ for any  two extreme points $u,v$ of  $B_\mathbb{X},$ it follows that $\| Tu\|=\|T\|$ for all extreme points $u$ of $B_\mathbb{X}.$ Let $S=\frac{T}{\|T\|}.$ Then $S$ maps each extreme point of $B_\mathbb{X} $ to some extreme point of $B_\mathbb{X}.$ Since $T$ is bijective, $S$ is bijective. Hence $S^{-1}$ exists and maps each extreme point of $B_\mathbb{X} $ to some extreme point of $B_\mathbb{X}.$ Now, $S^{-1}$ also attains its norm at some extreme point of $B_\mathbb{X}.$ So $\|S^{-1}\|=1.$ 
   	Hence,  $\|Sx\|\leq \|x\|$ for all $x\in\mathbb{X}$ and $\|x\|=\|S^{-1}(Sx)\|\leq \|Sx\|.$ This implies $\|Sx\|=\|x\|$ for all  $x\in\mathbb{X}.$ So $S$ is an isometry. Thus $T$ is a scalar multiple of an isometry. 
   	
   \end{proof}
   
   \begin{remark} (i)  Let $\mathbb{X}=\ell_{\infty}^n.$ Then the set $Ext(B_\mathbb{X})$ is a $\mathcal K$-set for  $\mathbb{X}.$ In particular, it follows that $\kappa(\ell_{\infty}^n) \leq 2^{n-1}.$\\
   	
   	(ii) For a  Banach space  $\mathbb{X},$ two minimal $\mathcal K$-sets may not have same cardinality. As for example, consider the Banach space $\mathbb{X}=\ell_{\infty}^2.$ Then $A=\{(1,1),(1,-1)\} $ is a minimal $\mathcal K$-set. Also, $B=\{(1,1),(1,0),(0,1)\}$ is a minimal $\mathcal K$-set, because if $T\in L(\mathbb X)$ preserves Birkhoff-James orthogonality on $B$ then $T$ is a scalar multiple of an isometry. On the other hand, if we consider the set $B_1=\{(1,1),(1,0)\}$ then  $T(x,y)=(x,2y-x)$  preserves Birkhoff-James orthogonality on $B_1$ but 
   	$T$ is not a scalar multiple of an isometry. Again,  if we consider the set $B_2=\{(1,1),(0,1)\}$ then  $T(x,y)=(2x-y,y)$  preserves Birkhoff-James orthogonality on $B_2$ but $T$ is not a scalar multiple of an isometry. Finally, if we consider the set $B_3=\{(1,0),(0,1)\}$ then  $T(x,y)=(x,0)$  preserves Birkhoff-James orthogonality at $B_3$ but $T$ is not a scalar multiple of an isometry. 
   \end{remark}

As the final result of this article, we will obtain a a refinement of the Blanco-Koldobsky-Turnsek Theorem for $ \ell_{1}^n, $ by using the following lemma.

   \begin{lemma}\label{nel2}
   	Let $\mathbb{X}=\ell_1^n.$ If $T\in \mathbb L(\mathbb{X})$ preserves  Birkhoff-James orthogonality at two extreme points $u_1,u_2$ of $B_{\mathbb{X}},$ then $\|Tu_1\|=\|Tu_2\|.$
   \end{lemma}

   \begin{proof}
   	The proof is analogous to  Lemma \ref{nel}, and is therefore omitted.
   \end{proof}

   \begin{theorem}\label{l1iso}
   	Let $\mathbb{X}=\ell_1^n.$
   		\begin{itemize}
   		\item[(i)] If $T\in \mathbb L(\mathbb{X})$ preserves  Birkhoff-James orthogonality at all extreme points of $B_{\mathbb{X}}$ then $T$ is a scalar multiple of an isometry.
   	\item[(ii)] $\kappa(\mathbb {X})=n.$ 
   \end{itemize}
   \end{theorem}
   \begin{proof}
   	(i) Let $T\in \mathbb L(\mathbb X)$ be such that $T$ preserves Birkhoff-James orthogonality at all extreme points of $B_{\mathbb X}.$ We show that $T$ is a scalar multiple of an isometry. If $T=0$ then the result follows trivially. Assume that $T\neq0.$ Let $A=\{e_i=(e_i(1),e_i(2),\dots,e_i(n)):1\leq i\leq n\}\subset Ext(B_{\mathbb{X}}),$ where for each $1\leq i,j\leq n,$
   	\begin{eqnarray*}
   		e_i(j)&=&0, \,j\neq i\\
   		&=&1,\, j=i .
   	\end{eqnarray*} 
   	Then it is easy to see that $A$ is an Auerbach basis for $\mathbb X.$ 
   	We claim that  $T$ is bijective. On the contrary suppose that $T$ is not bijective. Then there exists $i\in\{1,2,\dots,n \}$ such that $Te_i\in span~(T(A)\backslash \{Te_i\}).$ Then $Te_i=\sum_{j=1,j\neq i}^{n}\beta_jTe_j$ for some $\beta_j\in\mathbb R .$ This implies that $Te_i=T(\sum_{j=1,j\neq i}^{n}\beta_je_j).$ Since $A$ is an Auerbach basis, it follows $e_i\perp_B \sum_{j=1,j\neq i}^{n}\beta_je_j.$ Since $T$ preserves Birkhoff-James orthogonality at $e_i$, it follows That $Te_i\perp_B T(\sum_{j=1,j\neq i}^{n}\beta_je_j).$ This implies that $Te_i=0.$ Applying  Lemma \ref{nel2}, we conclude that $Te_k=0$ for all $1\leq k\leq n.$ This  contradicts the hypothesis that $T\neq 0.$ Thus, our claim is established.\\
   	Now, it follows from Corollary \ref{ioe} that the image of each extreme point of $B_\mathbb{X} $ is a scalar multiple of some extreme point of $B_\mathbb{X} .$ Since $\mathbb{X}$ is finite dimensional, $T$  attains its norm at an extreme point of $B_\mathbb{X}.$ Since $\| Tu\|=\|Tv\|$ for any two extreme points $u,v$ of  $B_\mathbb{X},$ it follows that $\| Tu\|=\|T\|$ for all extreme points $u$ of $B_\mathbb{X}.$ Let $S=\frac{T}{\|T\|}.$ Then $S$ maps each extreme point of $B_\mathbb{X} $ to some extreme point of $B_\mathbb{X}.$ Since $T$ is bijective, $S$ is bijective. Hence $S^{-1}$ exists and maps each extreme point of $B_\mathbb{X} $ to some extreme point  of $B_\mathbb{X}.$ Now, $S^{-1}$ attains its norm at some extreme point of $B_\mathbb{X}.$ So $\|S^{-1}\|=1.$ Hence, for all $x\in\mathbb{X},$ $\|Sx\|\leq \|x\|$ and $\|x\|=\|S^{-1}(Sx)\|\leq \|Sx\|.$ This implies $\|Sx\|=\|x\|$ for all  $x\in\mathbb{X}.$ So $S$ is an isometry. Thus $T$ is a scalar multiple of an isometry. 
   	
   	\medskip 
   	(ii) Since $\mathbb X $ has exactly $2n$ number of extreme points of the form $ \{ \pm e_1, \pm e_2, \ldots, \pm e_n\} $ so the set $A = \{ e_1, e_2, \ldots, e_n \} $ is a minimal  $\mathcal K$-set. So from Proposition \ref{dokn} it follows that $\kappa(\mathbb X) = n.$ 
   \end{proof}

We end this article with the following problem that remains unsolved:\\

\textbf{Open Question:} Characterize the minimal  $\mathcal K$-sets in $  \ell_{\infty}^n$ and hence find $ \kappa(\ell_{\infty}^n). $

\end{document}